\documentclass{article}

\usepackage[utf8]{inputenc}
\usepackage[T1]{fontenc}
\usepackage{hyperref}
\usepackage{url}
\usepackage{booktabs}
\usepackage{amsfonts}
\usepackage{nicefrac}
\usepackage{microtype}
\usepackage{lipsum}
\usepackage{graphicx}
\usepackage{amsmath}
\usepackage{amssymb}
\usepackage{amsthm}
\usepackage{enumitem}
\usepackage[all]{xy}
\SelectTips{cm}{12}

\newtheorem{lemma}{Lemma}

\newtheorem{theorem}{Theorem}
 
\newtheorem*{theorem*}{Theorem} 
\newtheorem*{corollary*}{Corollary}

\theoremstyle{remark}

\DeclareMathOperator\mat{M}
\DeclareMathOperator\elin{E}
\DeclareMathOperator\eorth{EO}
\DeclareMathOperator\stlin{St}
\DeclareMathOperator\storth{StO}
\DeclareMathOperator\stmap{st}
\DeclareMathOperator\diag{D}
\DeclareMathOperator\glin{GL}
\DeclareMathOperator\orth{O}

\newcommand{\op}{{\mathrm{op}}}
\DeclareMathOperator{\id}{id}
\DeclareMathOperator{\Aut}{Aut}
\DeclareMathOperator{\End}{End}

\newcommand{\up}[2]{{^{#1}\!{#2}}}

\makeatletter
\newcommand{\bigperp}{\mathop{\mathpalette\bigp@rp\relax}\displaylimits}
\newcommand{\bigp@rp}[2]{\vcenter{\m@th\hbox{\scalebox{\ifx#1\displaystyle2.1\else1.5\fi}{\(#1\perp\)}}}}
\makeatother

\newcommand{\Set}{\mathbf{Set}}
\newcommand{\Group}{\mathbf{Grp}}
\DeclareMathOperator{\Pro}{Pro}

\title{Another presentation\\ of orthogonal Steinberg groups}

\author{
  Egor Voronetsky
  \thanks{The work was supported by the Theoretical Physics and Mathematics Advancement Foundation ``BASIS'' and by ``Native towns'', a social investment program of PJSC ``Gazprom Neft''.}
  \\
  Chebyshev Laboratory, \\
  St. Petersburg State University, \\
  14th Line V.O., 29B, \\
  Saint Petersburg 199178 Russia \\
}

\begin{document}
\maketitle

\begin{abstract}
We use the pro-group approach to show that \(\storth(M, q)\) admits van der Kallen's ``another presentation'', where \(M\) is a module over a commutative ring with sufficiently isotropic quadratic form \(q\). Moreover, we construct an analog of ESD-transvections in orthogonal Steinberg pro-groups under some assumptions on their parameters.
\end{abstract}

\section{Introduction}

In \cite{AnotherPresentation} W. van der Kallen proved that the linear Steinberg group \(\stlin(n, K)\) over a commutative ring \(K\) is a central extension of the elementary linear group \(\elin(n, K)\) if \(n \geq 4\). More precisely, he showed that the Steinberg group admits a more invariant presentation and actually is a crossed module over the general linear group \(\glin(n, K)\). This was generalized by M. Tulenbaev in \cite{Tulenbaev} for linear groups over almost commutative rings.

For symplectic groups the same result was proved by A. Lavrenov in \cite{CentralityC}. He used essentially the same approach: there is another presentation of the symplectic Steinberg group \(\mathrm{StSp}(2\ell, K)\) over a commutative ring \(K\) for \(\ell \geq 3\) such that it is obvious that this group is a central extension of the elementary symplectic group \(\mathrm{ESp}(2\ell, K)\). Together with S. Sinchuk he also proved centrality of the corresponding \(K_2\)-functors for the Chevalley groups of the types \(\mathsf D_\ell\) for \(\ell \geq 3\) and \(\mathsf E_\ell\) in \cite{CentralityD, CentralityE} using a different method.

Another presentation for orthogonal groups was proved by S. B\"oge in \cite{AnotherPresentationOrth}, but only for fields of characteristic not \(2\). She also considered sufficiently isotropic orthogonal groups instead of the split ones.

In \cite{LinK2} we reproved that \(\stlin(n, K)\) is a crossed module over \(\glin(n, K)\) using pro-groups. This more powerful method allowed to generalize the result for isotropic linear groups over almost commutative rings and for matrix linear groups over non-commutative rings satisfying a local stable rank condition. Together with Lavrenov and Sinchuk we applied the pro-group approach to the simple simply connected Chevalley groups of rank at least \(3\) in \cite{ChevK2}. The same result for isotropic odd unitary groups (including the Chevalley groups of the types \(\mathsf A_\ell\), \(\mathsf B_\ell\), \(\mathsf C_\ell\), and \(\mathsf D_\ell\)) was proved in \cite{UnitK2}, see also \cite{Thesis}. For the Chevalley groups of rank \(2\) there are counterexamples, see \cite{Wendt}.

The pro-group approach does not give any ``another presentation'' of the corresponding Steinberg group by itself. In this paper we prove the following:
\begin{theorem*}
Let \(K\) be a unital commutative ring, \(M\) be a finitely presented \(K\)-module with a quadratic form \(q\) and pairwise orthogonal hyperbolic pairs \((e_{-1}, e_1), \ldots, (e_{-\ell}, e_\ell)\) for \(\ell \geq 3\). Then the orthogonal Steinberg group \(\storth(M, q)\) is isomorphic to the abstract group \(\storth^*(M, q)\) generated by elements \(X^*(u, v)\), where \(u, v \in M\) are vectors, \(u\) lies in the orbit of \(e_1\) under the action of the orthogonal group, and \(u \perp v\). The relations on these elements are
\begin{itemize}
\item \(X^*(u, v + v') = X^*(u, v)\, X^*(u, v')\);
\item \(X^*(u, v)\, X^*(u', v')\, X^*(u, v)^{-1} = X^*(T(u, v)\, u', T(u, v)\, v')\), where \(T(u, v)\) is the corresponding ESD-transvection;
\item \(X^*(u, va) = X^*(v, -ua)\) if \((u, v)\) lies in the orbit of \((e_1, e_2)\);
\item \(X^*(u, ua) = 1\).
\end{itemize}
\end{theorem*}

Using this variant of ``another presentation'', it is obvious that \(\storth(M, q)\) is a crossed module over \(\orth(M, q)\).

During the proof we also give a general definition of ESD-transvections \(X(u, v) \in \storth(M, q)\), where \(u\) is a member of a hyperbolic pair and \(v \perp u\). These elements lift ordinary ESD-transvections \(T(u, v) \in \orth(M, q)\) and satisfy various identities, but their existence for all such \(u\) is unclear.

The author wants to express his gratitude to Nikolai Vavilov, Sergey Sinchuk and Andrei Lavrenov for motivation and helpful discussions. He also wants to thank Pavel Gvozdevsky for useful comments.

\section{Orthogonal Steinberg pro-groups}

We use the group-theoretical notation \(\up gh = ghg^{-1}\) and \([g, h] = ghg^{-1}h^{-1}\). If a group \(G\) acts on a group \(H\) by automorphisms, we denote the action by \(\up gh\).

Let \(K\) be a commutative unital ring, \(M\) be an \(K\)-module of finite presentation with a quadratic form \(q \colon M \to K\) and the associated symmetric bilinear form \(\langle m, m' \rangle = q(m + m') - q(m) - q(m)\). The orthogonal group \(\orth(M, q)\) consists of linear automorphisms \(g \in \glin(M)\) such that \(q(gm) = q(m)\) for all \(m\). A hyperbolic pair \((u, v)\) in \(M\) is a pair of vectors with the properties \(\langle u, v \rangle = 1\) and \(q(u) = q(v) = 0\).

Consider vectors \(u, v \in M\) such that \(q(u) = \langle u, v \rangle = 0\). In this case the operator
\[T(u, v) \colon M \to M,\, m \mapsto m + u\, \langle v, m \rangle - v\, \langle u, m \rangle - u\, q(v)\, \langle u, m \rangle\]
is called an Eichel -- Siegel -- Dickson transvection (or an ESD-transvection) with parameters \(u\) and \(v\), see \cite{OddPetrov} for details in a more general context. The following lemma summarizes the well-known properties of these operators, all of them may be checked by direct calculations.

\begin{lemma}\label{esd-facts}
Each ESD-transvection lies in the orthogonal group \(\orth(M, q)\). Moreover,
\begin{itemize}
\item \(T(u, v)\, T(u, v') = T(u, v + v')\) for \(q(u) = \langle u, v \rangle = \langle u, v' \rangle = 0\);
\item \(T(ua, v) = T(u, va)\) for \(q(u) = \langle u, v \rangle = 0\) and \(a \in K\);
\item \(\up g{T(u, v)} = T(gu, gv)\) for \(q(u) = \langle u, v \rangle = 0\) and \(g \in \orth(M, q)\);
\item \(T(u, v) = T(v, -u)\) for \(q(u) = q(v) = \langle u, v \rangle = 0\);
\item \(T(u, ua) = 1\) for \(q(u) = 0\) and \(a \in K\).
\end{itemize}
\end{lemma}

From now on suppose that there are pairwise orthogonal hyperbolic pairs \((e_{-1}, e_1), \ldots, (e_{-\ell}, e_\ell)\) in \(M\) for \(\ell \geq 3\). We denote the orthogonal complement of the span of these vectors by \(M_0\).

Recall that elementary orthogonal transvections of long root type are the elements \(t_{ij}(a) = T(e_i, e_{-j} a) = T(e_{-j}, -e_i a)\) for \(i \neq \pm j\) and \(a \in K\), the elementary orthogonal transvections of short root type are the elements \(t_j(m) = T(e_{-j}, -m)\) for \(m \in M_0\). These elements are precisely the elementary transvections in \(\orth(M, q)\) considered as an isotropic odd unitary group, but with slightly different parametrizations. See \cite{OddPetrov} or \cite{StabVor} for details. Clearly,
\[T(e_i, m) = t_{-i}(-m_0)\, \prod_{j \neq \pm i} t_{i, -j}(m_j)\]
if \(\langle e_i, m \rangle = 0\), where \(m = m_0 + \sum_{1 \leq |j| \leq \ell} e_j m_j\) and \(m_0 \in M_0\).

The orthogonal Steinberg group \(\storth(M, q)\) is the abstract group generated by the elements \(x_{ij}(a)\) and \(x_j(m)\) for \(a \in K\), \(m \in M_0\), and \(i \neq \pm j\). The relations on these elements are the following:
\begin{itemize}
\item \(x_{ij}(a + b) = x_{ij}(a)\, x_{ij}(b)\);
\item \(x_{ij}(a) = x_{-j, -i}(-a)\);
\item \(x_i(m + m') = x_i(m)\, x_i(m')\);
\item \([x_{ij}(a), x_{kl}(b)] = 1\) for \(j \neq k \neq -i \neq -l \neq j\);
\item \([x_{ij}(a), x_{j, -i}(b)] = 1\);
\item \([x_{ij}(a), x_{jk}(b)] = x_{ik}(ab)\) for \(i \neq \pm k\);
\item \([x_i(m), x_j(m')] = x_{-i, j}(-\langle m, m' \rangle)\) for \(i \neq \pm j\);
\item \([x_i(m), x_{jk}(a)] = 1\) for \(j \neq i \neq -k\);
\item \([x_i(m), x_{ij}(a)] = x_{-i, j}(-q(m)\, a)\, x_j(ma)\).
\end{itemize}

There is a group homomorphism \(\stmap \colon \storth(M, q) \to \orth(M, q), x_{ij}(a) \mapsto t_{ij}(a), x_j(m) \mapsto t_j(m)\). Its image is the elementary orthogonal group \(\eorth(M, q)\), i.e. the subgroup of \(\orth(M, q)\) generated by the elementary orthogonal transvections.

We would like to find analogs of ESD-transvections in the Steinberg group \(\storth(M, q)\). In order to do so, we use results from \cite{UnitK2} or \cite{Thesis} (if \(q\) is split of even rank, then it suffices to use the case of \(\mathsf D_\ell\) considered in \cite{ChevK2}). From now on we assume that the reader is familiar with the notion of categorical pro-completion. Recall that there is a ``forgetful'' functor from the category of pro-groups \(\Pro(\Group)\) to the category of group objects in the category of pro-sets \(\Pro(\Set)\). This functor is fully faithful, so we identify pro-groups with the corresponding pro-sets. The projective limits in \(\Pro(\Set)\) are denoted by \(\varprojlim^{\Pro}\) and various pro-sets are labeled with the upper index \((\infty)\), such as \(X^{(\infty)}\). Both categories \(\Pro(\Group)\) and \(\Pro(\Set)\) are complete and cocomplete.

We also use the following convention from \cite{ChevK2, LinK2, UnitK2}: if a morphism between pro-sets is given by a first order term (possibly multi-sorted), then we add the upper index \((\infty)\) for the formal variables. For example, \([g^{(\infty)}, h^{(\infty)}]\) denotes the commutator morphism \(G^{(\infty)} \times G^{(\infty)} \to G^{(\infty)}\) for a pro-group \(G^{(\infty)}\), and \(m^{(\infty)} a^{(\infty)}\) denotes the product morphism \(M^{(\infty)} \times K^{(\infty)} \to M^{(\infty)}\) for a pro-ring \(K^{(\infty)}\) and its pro-module \(M^{(\infty)}\). The domains of such variables are usually clear from the context.

For any \(s \in K\) an \(s\)-homotope of \(K\) is the commutative non-unital \(K\)-algebra \(K^{(s)} = \{a^{(s)} \mid a \in K\}\). The operations are given by
\[
a^{(s)} + b^{(s)} = (a + b)^{(s)}, \quad
a^{(s)} b = (ab)^{(s)}, \quad
a^{(s)} b^{(s)} = (abs)^{(s)}.
\]
Similarly, \(M^{(s)} = \{m^{(s)} \mid m \in M\}\) is a non-unital \(K^{(s)}\)-module and a unital \(K\)-module with the operations
\[
m^{(s)} + {m'}^{(s)} = (m + m')^{(s)}, \quad
m^{(s)} a = (ma)^{(s)}, \quad
m^{(s)} a^{(s)} = (mas)^{(s)}.
\]
The forms on \(M\) may be extended to \(M^{(s)}\) by
\[
q\bigl(m^{(s)}\bigr) = (q(m)\, s)^{(s)}, \quad
\bigl\langle m^{(s)}, {m'}^{(s)} \bigr\rangle = (\langle m, m' \rangle\, s)^{(s)}.
\]

Note that in the definition of \(\storth(M, q)\) we do not need the unit of \(K\). Hence we define \(\storth^{(s)}(M, q) = \storth(M^{(s)}, q)\) by the same generators and relations, but the parameters are taken from \(K^{(s)}\) and \(M_0^{(s)}\). If \(s, s' \in K\), then there are homomorphisms
\begin{align*}
K^{(ss')} \to K^{(s)}, a^{(ss')} \mapsto (as')^{(s)},\\
M^{(ss')} \to M^{(s)}, m^{(ss')} \mapsto (ms')^{(s)},
\end{align*}
and the obvious homomorphism \(\storth^{(ss')}(M, q) \to \storth^{(s)}(M, q)\) of Steinberg groups.

Fix a multiplicative subset \(S \subseteq K\). The formal projective limit \(K^{(\infty, S)} = \varprojlim^{\Pro}_{s \in S} K^{(s)}\) is actually a commutative non-unital pro-\(K\)-algebra, \(M^{(\infty, S)} = \varprojlim^{\Pro}_{s \in S} M^{(s)}\) is a pro-\(K\)-module. Similarly, the Steinberg pro-group \(\storth^{(\infty, S)}(M, q) = \varprojlim^{\Pro}_{s \in S} \storth^{(s)}(M, q)\) is indeed a pro-group. If \(S = K \setminus \mathfrak p\) for a prime ideal \(\mathfrak p\), then we write \(K^{(\infty, \mathfrak p)}\), \(M^{(\infty, \mathfrak p)}\), and \(\storth^{(\infty, \mathfrak p)}(M, q)\) instead of \(K^{(\infty, S)}\), \(M^{(\infty, S)}\), and \(\storth^{(\infty, S)}(M, q)\). Similarly, if \(S = \{1, f, f^2, \ldots\}\) for \(f \in K\), then we write \(f\) instead of \(S\) in the indices. Finally, the constructions \(K^{(\infty, S)}\), \(M^{(\infty, S)}\), and \(\storth^{(\infty, S)}(M, q)\) are contravariant functors on \(S\), in the case \(S = \{1\}\) we get \(K\), \(M\), and \(\storth(M, q)\) up to canonical isomorphisms. There are well-defined morphisms \(x_{ij} \colon K^{(\infty, S)} \to \storth^{(\infty, S)}(M, q)\) and \(x_j \colon M_0^{(\infty, S)} \to \storth^{(\infty, S)}(M, q)\) of pro-groups, they generate the Steinberg pro-group in the categorical sense by \cite[lemma 5]{UnitK2} or \cite[beginning of section 3.3]{Thesis}.

In order to apply the main results from \cite{UnitK2} or \cite{Thesis}, we need a lemma.

\begin{lemma}\label{quasi-finite}
Let \((R, \Delta)\) be the odd form \(K\)-algebra constructed by \((M, q)\) as in \cite{StabVor} or \cite[section 1.1]{Thesis}. Then \(R\) is locally finite, i.e. any finite subset of \(R\) is contained in a finite subalgebra. The construction of \((R, \Delta)\) commutes with localizations.
\end{lemma}
\begin{proof}
Here we use that \(M\) is of finite presentation. Since \(M\) is of finite type, the \(K\)-algebra \(\End_K(M)\) is locally finite as a factor-algebra of the subalgebra \(\{x \in \mat(n, K) \mid xN \leq N\}\) of \(\mat(n, K)\), where \(0 \to N \to K^n \to M \to 0\) is a short exact sequence. Then \(R\) is also locally finite, by definition it is the algebra
\[\{(x^\op, y) \in \End_K(M)^\op \times \End_K(M) \mid \langle xm, m' \rangle = \langle m, ym' \rangle \text{ for all } m, m' \in M\}.\]

Now recall that \(M\) is of finite presentation. This implies that the construction  \(\End_K(M)\) commutes with localizations, i.e. the map \(S^{-1} \End_K(M) \to \End_{S^{-1} K}(S^{-1} M)\) is an isomorphism for all multiplicative subsets \(S \subseteq K\). Hence the construction of \(R\) also commutes with localizations. Finally,
\[\Delta = \{(x^\op, y; z^\op, w) \in R \times R \mid xy + z + w = 0, q(ym) + \langle m, wm \rangle = 0 \text{ for all } m \in M\}.\]
Its localization with respect to \(S \subseteq K\) is
\[\textstyle S^{-1} \Delta = \bigl\{\bigl(\frac rs, \frac{r'}{s^2}\bigr) \mathbin{\big|} (r, r') \in \Delta\bigr\} \subseteq S^{-1} R \times S^{-1} R\]
by definition from \cite{UnitK2} or \cite[section 2.2]{Thesis}, so the construction of \(\Delta\) also commutes with localizations. Note that in the definition of \(\Delta\) is suffices to take \(m\) from a generating set of \(M\).
\end{proof}

Now take a multiplicative subset \(S \subseteq K\). Each \(\frac b s \in S^{-1} K\) gives endomorphisms of pro-groups \(K^{(\infty, S)}\) and \(M^{(\infty, S)}\) by \(a^{(ss')} \mapsto (ab)^{(s')}\) and \(m^{(ss')} \mapsto (mb)^{(s')}\). Similarly, each \(\frac{m'}{s} \in S^{-1} M\) gives the morphisms \(K^{(\infty, S)} \to M^{(\infty, S)}, a^{(ss')} \mapsto (m' a)^{(s')}\) and \(M^{(\infty, S)} \to K^{(\infty, S)}, m^{(ss')} \mapsto \langle m, m' \rangle^{(s')}\). These morphisms are denoted by the terms \(a^{(\infty)} \frac b s = \frac b s a^{(\infty)}\), \(m^{(\infty)} \frac b s\), \(\frac{m'}{s} a^{(\infty)}\), and \(\bigl\langle m^{(\infty)}, \frac{m'}{s} \bigr\rangle = \bigl\langle \frac{m'}{s}, m^{(\infty)} \bigr\rangle\) with the free variables \(a^{(\infty)}\) and \(m^{(\infty)}\). By \cite[beginning of section 10]{UnitK2} or \cite[lemma 26]{Thesis}, the local Steinberg group \(\storth(S^{-1} M, q)\) acts on the corresponding Steinberg pro-group \(\storth^{(\infty, S)}(M, q)\) by automorphisms. The action is given by the obvious formulas on generators when there is a corresponding Steinberg relation:
\begin{itemize}
\item \(\up{x_{ij}(a / s)}{x_{kl}\bigl(b^{(\infty)}\bigr)} = x_{kl}\bigl(b^{(\infty)}\bigr)\) for \(i \neq l \neq -j \neq -k \neq i\);
\item \(\up{x_{ij}(a / s)}{x_{j, -i}\bigl(b^{(\infty)}\bigr)} = x_{j, -i}\bigl(b^{(\infty)}\bigr)\);
\item \(\up{x_{ij}(a / s)}{x_{jk}\bigl(b^{(\infty)}\bigr)} = x_{ik}\bigl(\frac a s b^{(\infty)}\bigr)\, x_{jk}\bigl(b^{(\infty)}\bigr)\) for \(i \neq \pm k\);
\item \(\up{x_{ij}(a / s)}{x_k\bigl(m^{(\infty)}\bigr)} = x_k\bigl(m^{(\infty)}\bigr)\) for \(i \neq k \neq -j\);
\item \(\up{x_{ij}(a / s)}{x_i\bigl(m^{(\infty)}\bigr)} = x_{-i, j}\bigl(q\bigl(m^{(\infty)}\bigr)\, \frac a s\bigr)\, x_j\bigl(-m^{(\infty)} \frac a s\bigr)\, x_i\bigl(m^{(\infty)}\bigr)\);
\item \(\up{x_i(m / s)}{x_j\bigl({m'}^{(\infty)}\bigr)} = x_{-i, j}\bigl(-\bigl\langle \frac m s, {m'}^{(\infty)} \bigr\rangle\bigr)\, x_j\bigl({m'}^{(\infty)}\bigr)\) for \(i \neq \pm j\);
\item \(\up{x_i(m / s)}{x_{jk}\bigl(a^{(\infty)}\bigr)} = x_{jk}\bigl(a^{(\infty)}\bigr)\) for \(j \neq i \neq -k\);
\item \(\up{x_i(m / s)}{x_{ij}\bigl(a^{(\infty)}\bigr)} = x_{-i, j}\bigl(-q\bigl(\frac m s\bigr)\, a^{(\infty)}\bigr)\, x_j\bigl(\frac m s a^{(\infty)}\bigr)\, x_{ij}\bigl(a^{(\infty)}\bigr)\).
\end{itemize}
Note that it is not a morphism \(\storth(S^{-1} M, q) \times \storth^{(\infty, S)}(M, q) \to \storth^{(\infty, S)}(M, q)\) of pro-sets, but a homomorphism \(\storth(S^{-1} M, q) \to \Aut_{\Pro(\Group)}\bigl(\storth^{(\infty, S)}(M, q)\bigr)\) of abstract groups.

This action is extranatural on \(S\). More precisely, if \(S \subseteq S'\) are two multiplicative subsets, \(g \in \storth(S^{-1} M, q)\) and \(g'\) is its image in \(\storth({S'}^{-1} M, q)\), then the following diagram with canonical vertical arrows commutes:
\[\xymatrix@R=30pt@C=120pt@!0{
\storth^{(\infty, S')}(M, q) \ar[r]^{\up {g'}{(-)}} \ar[d] & \storth^{(\infty, S')}(M, q) \ar[d]\\
\storth^{(\infty, S)}(M, q) \ar[r]^{\up g{(-)}} & \storth^{(\infty, S)}(M, q).
}\]

By lemma \ref{quasi-finite} and \cite[theorem 2]{UnitK2} (or \cite[lemma 26]{Thesis}), for any prime ideal \(\mathfrak p\) of \(K\) the action of \(\storth(M_{\mathfrak p}, q)\) on \(\storth^{(\infty, \mathfrak p)}(M, q)\) continues to an action of \(\orth(M_{\mathfrak p}, q)\). Also there is an action of \(\orth(M, q)\) on \(\storth(M, q)\) such that \(\stmap \colon \storth(M, q) \to \orth(M, q)\) is a crossed module, see \cite[theorem 3]{UnitK2} or \cite[theorem 7]{Thesis}. This action is also extranatural: if \(g \in \orth(M, q)\), then the induced automorphisms of \(\storth(M, q)\) and \(\storth^{(\infty, \mathfrak p)}(M, q)\) give a commutative square as above.

\section{Another presentation}

Recall the following well-known result.

\begin{lemma}\label{local-isotropy}
Suppose that \(K\) is local. Then the group \(\orth(M, q)\) acts transitively on the set of members of hyperbolic pairs and on the set of pairs \((u, v)\), where \(u, v\) are members of orthogonal hyperbolic pairs.
\end{lemma}
\begin{proof}
Actually, the orthogonal group acts transitively on the set of hyperbolic pairs and on the set of couples of orthogonal hyperbolic pairs. This is Witt's cancellation theorem, see \cite[corollary 8.3]{ASR} for a proof in our generality.
\end{proof}

Let \(S \subseteq K\) be a multiplicative subset and \(u \in S^{-1} M\) be a member of a hyperbolic pair. There is a well-defined split epimorphism \(\langle u, v^{(\infty)} \rangle \colon M^{(\infty, S)} \to K^{(\infty, S)}\) of pro-groups, i.e. its kernel \(M_u^{(\infty, S)}\) has a direct complement isomorphic to \(K^{(\infty, S)}\). If \(g \in \orth(S^{-1} M, q)\), then it naturally acts on \(M^{(\infty, S)}\) (since \(M\) is of finite presentation) and maps the direct summand \(M_u^{(\infty, S)}\) to \(M_{gu}^{(\infty, S)}\). Our goal is to construct morphisms \(X\bigl(u, v^{(\infty)}\bigr) \colon M_u^{(\infty, S)} \to \storth^{(\infty, S)}(M, q)\) of pro-groups satisfying various natural conditions.

\begin{lemma}\label{esd-local}
Let \(\mathfrak p\) be a prime ideal of \(K\). Then for all members \(u\) of hyperbolic pairs in \(M_{\mathfrak p}\) there are unique morphisms \(X\bigl(u, v^{(\infty)}\bigr) \colon M_u^{(\infty, \mathfrak p)} \to \storth^{(\infty, \mathfrak p)}(M, q)\) of pro-groups such that
\begin{enumerate}
\item \(\up g{X\bigl(u, v^{(\infty)}\bigr)} = X\bigl(gu, gv^{(\infty)}\bigr)\) for \(g \in \orth(M_{\mathfrak p}, q)\);
\item \(X\bigl(u, u a^{(\infty)}\bigr) = 1\);
\item \(X\bigl(ua, v^{(\infty)}\bigr) = X\bigl(u, v^{(\infty)} a\bigr)\) for \(a \in K_{\mathfrak p}^*\);
\item \(X\bigl(u, v a^{(\infty)}\bigr) = X\bigl(v, -u a^{(\infty)}\bigr)\) if \(u, v\) are members of orthogonal hyperbolic pairs;
\item \(X\bigl(e_i, e_j a^{(\infty)}\bigr) = x_{i, -j}\bigl(a^{(\infty)}\bigr)\) for \(i \neq \pm j\);
\item \(X\bigl(e_i, m^{(\infty)}\bigr) = x_{-i}\bigl(-m^{(\infty)}\bigr)\) for \(m^{(\infty)}\), where the formal parameter has the domain \(M_0^{(\infty, \mathfrak p)}\).
\end{enumerate}
\end{lemma}
\begin{proof}
Note that \(M_{e_\ell}^{(\infty, \mathfrak p)} = M_0^{(\infty, \mathfrak p)} \oplus \bigoplus_{i \neq -\ell} e_i K^{(\infty, \mathfrak p)}\). Let
\[X\bigl(e_\ell, v^{(\infty)}\bigr) = x_{-\ell}\bigl(v_0^{(\infty)}\bigr)\, \prod_{0 < i < \ell} \bigl(x_{\ell, -i}\bigl(v^{(\infty)}_i\bigr)\, x_{\ell i}\bigl(v^{(\infty)}_{-i}\bigr) \bigr) \colon M_{e_\ell}^{(\infty, \mathfrak p)} \to \storth^{(\infty, \mathfrak p)}(M, q),\]
where the variables \(v_0^{(\infty)}\) and \(v_i^{(\infty)}\) are the components of \(v^{(\infty)}\) in the decomposition of \(M_{e_\ell}^{(\infty, \mathfrak p)}\). This is the only possible choice satisfying the conditions \((2)\), \((5)\), and \((6)\).

The parabolic subgroup \(P = \{g \in \orth(M_{\mathfrak p}, q) \mid ge_\ell \in e_\ell K_{\mathfrak p}\}\) is clearly generated by the diagonal subgroup
\[\diag(M_{\mathfrak p}, q) = \{g \in \orth(M_{\mathfrak p}, q) \mid ge_i \in e_i K^*_{\mathfrak p} \text{ for all } 0 < |i| \leq \ell\},\]
the unipotent radical
\[U = \langle t_{\ell i}(K_{\mathfrak p}), t_{-\ell}(M_{0, \mathfrak p}) \mid 0 < |i| < \ell \rangle,\]
and the smaller orthogonal group
\[\textstyle\orth(M_{0, \mathfrak p} \oplus \bigoplus_{0 < |i| < \ell} e_i K_{\mathfrak p}, q)\]
trivially acting on \(e_{-\ell}\) and \(e_\ell\). By \cite[theorem 5]{OddPetrov}, this smaller orthogonal group is generated by its elementary orthogonal transvections, its intersection with the diagonal subgroup, and the orthogonal group
\[O = \orth(M_{0, \mathfrak p} \oplus e_{-1} K_{\mathfrak p} \oplus e_1 K_{\mathfrak p}, q)\]
trivially acting on \(e_i\) for \(|i| > 1\).

Now we show that
\[\textstyle \up g{X\bigl(e_\ell, v^{(\infty)}\bigr)} = X\bigl(e_\ell, g v^{(\infty)} \frac{ge_\ell}{e_\ell}\bigr)\]
for every such a generator \(g \in P\), where \(\frac{ge_\ell}{e_\ell}\) is the element of \(K_{\mathfrak p}^*\) such that \(g e_\ell = e_\ell \frac{ge_\ell}{e_\ell}\). Indeed, the diagonal subgroup acts by roots on both sides (see \cite[formulas (Ad1)--(Ad6)]{UnitK2}), for \(O\) this follows from the same argument using \cite[proposition 2]{UnitK2} or \cite[theorem 5]{Thesis} to eliminate the hyperbolic pair \((e_{-1}, e_1)\) and enlarge \(M_0\), and for the elementary transvections this may be checked directly by the definitions.

Let \(u \in M_{\mathfrak p}\) be a member of a hyperbolic pair. By lemma \ref{local-isotropy}, there is \(g \in \orth(M_{\mathfrak p}, q)\) such that \(u = ge_\ell\). Let \(X\bigl(u, v^{(\infty)}\bigr) = \up g{X\bigl(e_\ell, g^{-1} v^{(\infty)}\bigr)}\), this morphism is independent on \(g\) and satisfies \((1)\) and \((2)\). To prove \((3)\), note that there is \(h \in \diag(M_{\mathfrak p}, q)\) such that \(h e_\ell = e_\ell a\). Hence
\[\textstyle X\bigl(e_\ell a, v^{(\infty)}\bigr) = \up h{X\bigl(e_\ell, h^{-1} v^{(\infty)}\bigr)} = X\bigl(e_\ell, v^{(\infty)} \frac{he_\ell}{e_\ell}\bigr) = X\bigl(e_\ell, v^{(\infty)} a\bigr),\]
and \((3)\) for other vectors \(u\) follows from the definition of \(X\bigl(u, v^{(\infty)}\bigr)\).

The properties \((5)\) and \((6)\) follow by applying various permutation matrices from \(\orth(M_{\mathfrak p}, q)\) to \(x_{\ell, \ell - 1}\bigl(a^{(\infty)}\bigr)\) and \(x_\ell\bigl(m^{(\infty)}\bigr)\). Finally, in order to prove \((4)\) we again use lemma \ref{local-isotropy}. Without loss of generality, \(u = e_\ell\) and \(v = e_{\ell - 1}\). Then
\[X\bigl(e_\ell, e_{\ell - 1} a^{(\infty)}\bigr) = x_{\ell, 1 - \ell}\bigl(a^{(\infty)}\bigr) = x_{\ell - 1, -\ell}\bigl(-a^{(\infty)}\bigr) = X\bigl(e_{\ell - 1}, -e_\ell a^{(\infty)}\bigr).\qedhere\]
\end{proof}

In order to globalize such transvections, we need a lemma similar to \cite[lemma 27]{Thesis}.

\begin{lemma}\label{module-costalks}
Let \(S \subseteq K\) be a multiplicative subset, \(N\) be arbitrary \(K\)-module. Consider the family \(\pi_{\mathfrak p} \colon N^{(\infty, \mathfrak p)} \to N^{(\infty, S)}\) of canonical pro-group morphisms, where \(\mathfrak p\) runs over all prime ideals of \(K\) disjoint with \(S\). Then this family generates \(N^{(\infty, S)}\) in the following sense: for every pro-group \(G^{(\infty)} \in \Pro(\Group)\) the map
\[\Pro(\Group)\bigl(N^{(\infty, S)}, G^{(\infty)}\bigr) \to \prod_{\mathfrak p} \Pro(\Group)\bigl(N^{(\infty, \mathfrak p)}, G^{(\infty)}\bigr)\]
is injective.
\end{lemma}
\begin{proof}
If suffices to consider an ordinary group \(G\). Take group homomorphisms \(f, g \colon N^{(s)} \to G\) for some \(s \in S\) such that \(f \circ \pi_{\mathfrak p} = g \circ \pi_{\mathfrak p} \colon N^{(\infty, \mathfrak p)} \to G\) for all prime ideals \(\mathfrak p\) disjoint with \(S\). The set
\[\mathfrak a = \{a \in K \mid f|_{N^{(s)} a} = g|_{N^{(s)} a}\}\]
is an ideal of \(K\). By assumption it is not contained in any such \(\mathfrak p\), so there exists \(s' \in \mathfrak a \cap S\). This means that \(f|_{N^{(ss')}} = g|_{N^{(ss')}}\).
\end{proof}

\begin{theorem}\label{esd-identities}
Let \(K\) be a unital commutative ring, \(M\) be a finitely presented \(K\)-module with a quadratic form \(q\) and pairwise orthogonal hyperbolic pairs \((e_{-1}, e_1), \ldots, (e_{-\ell}, e_\ell)\) for \(\ell \geq 3\). Let also \(S \subseteq K\) be a multiplicative subset and \(u \in S^{-1} M\) be a member of a hyperbolic pair. Then there exists at most one morphism \(X\bigl(u, v^{(\infty)}\bigr) \colon M^{(\infty, S)}_u \to \storth^{(\infty, S)}(M, q)\) of pro-groups such that for any prime ideal \(\mathfrak p\) of \(K\) disjoint with \(S\) the following diagram with canonical vertical arrows is commutative:
\[\xymatrix@R=30pt@C=120pt@!0{
M_u^{(\infty, \mathfrak p)} \ar[r]^(0.4){X(u, v^{(\infty)})} \ar[d] & \storth^{(\infty, \mathfrak p)}(M, q) \ar[d]\\
M_u^{(\infty, S)} \ar[r]^(0.4){X(u, v^{(\infty)})} & \storth^{(\infty, S)}(M, q).
}\]
These morphisms also satisfy the following identities:
\begin{enumerate}
\item \(\up g{X\bigl(u, v^{(\infty)}\bigr)} = X\bigl(gu, gv^{(\infty)}\bigr)\) for \(g \in \storth(S^{-1} M, q)\) if the left hand side is defined;
\item \(X\bigl(u, u a^{(\infty)}\bigr) = 1\) if the left hand side is defined;
\item \(X\bigl(ua, v^{(\infty)}\bigr) = X\bigl(u, v^{(\infty)} a\bigr)\) for \(a \in S^{-1} K^*\) if both sides are defined;
\item \(X\bigl(u, va^{(\infty)}\bigr) = X\bigl(v, -ua^{(\infty)}\bigr)\) if \(u, v \in S^{-1} M\) are members of orthogonal hyperbolic pairs and both sides are defined;
\item \(X\bigl(e_i, e_j a^{(\infty)}\bigr) = x_{i, -j}\bigl(a^{(\infty)}\bigr)\) for \(i \neq \pm j\);
\item \(X\bigl(e_i, m^{(\infty)}\bigr) = x_{-i}\bigl(-m^{(\infty)}\bigr)\).
\end{enumerate}
\end{theorem}
\begin{proof}
By lemma \ref{module-costalks} applied to \(M\) the pro-group \(M_u^{(\infty, S)}\) is generated by the canonical morphisms \(M_u^{(\infty, \mathfrak p)} \to M_u^{(\infty, S)}\) for all prime ideals \(\mathfrak p\) of \(K\) disjoint with \(S\). Hence the morphisms \(X\bigl(u, v^{(\infty)}\bigr)\) are unique whenever they exist. Clearly, such a morphism exists and satisfies \((5)\), \((6)\) for \(u = e_i\). The set of members of hyperbolic pairs \(u \in S^{-1} M\) such that \(X\bigl(u, v^{(\infty)}\bigr)\) exists is closed under the action of \(\storth(S^{-1} M, q)\) by extranaturality of the action. All other properties follow from lemmas \ref{esd-local} and \ref{module-costalks}.
\end{proof}

\begin{theorem}\label{esd-existence}
Let \(K\) be a unital commutative ring, \(M\) be a finitely presented \(K\)-module with a quadratic form \(q\) and pairwise orthogonal hyperbolic pairs \((e_{-1}, e_1), \ldots, (e_{-\ell}, e_\ell)\) for \(\ell \geq 3\). Let also \(S \subseteq K\) be a multiplicative subset, \(u \in S^{-1} M\) be a member of a hyperbolic pair. The morphism \(X\bigl(u, v^{(\infty)}\bigr) \colon M_u^{(\infty, S)} \to \storth^{(\infty, S)}(M, q)\) from theorem \ref{esd-identities} exists in the following cases:
\begin{enumerate}
\item \(S = K \setminus \mathfrak p\) for a prime ideal \(\mathfrak p\);
\item \(u\) lies in the orbit of \(e_1\) under the action of \(\eorth(S^{-1} M, q)\);
\item \(S = \{1\}\) and \(u\) lies in the orbit of \(e_1\) under the action of \(\orth(M, q)\).
\end{enumerate}
In the last case \(X(u, v) \in \storth(M, q)\) maps to \(T(u, v) \in \orth(M, q)\).
\end{theorem}
\begin{proof}
The first case is lemma \ref{esd-local}, the second is theorem \ref{esd-identities}. If \(S = \{1\}\) and \(u = ge_1\) for some \(g \in \orth(M, q)\), then let \(X(u, v) = \up g{X(e_1, g^{-1} v)}\) for all vectors \(v \perp u\). This morphism satisfies the definition from theorem \ref{esd-identities} by extranaturality of the action. The remaining claim follows from the uniqueness of \(X(u, v)\) and the definitions.
\end{proof}

Finally, we prove that the orthogonal Steinberg group admits ``another presentation'' in the sense of van der Kallen. Note that the third property of \(X(u, v)\) from theorem \ref{esd-identities} is not needed.

\begin{theorem}\label{another-presentation}
Let \(K\) be a unital commutative ring, \(M\) be a finitely presented \(K\)-module with a quadratic form \(q\) and pairwise orthogonal hyperbolic pairs \((e_{-1}, e_1), \ldots, (e_{-\ell}, e_\ell)\) for \(\ell \geq 3\). Let \(\storth^*(M, q)\) be the abstract group generated by elements \(X^*(u, v)\) for vectors \(u, v \in M\) such that \(u \in \orth(M, q)\, e_1\) and \(v \perp u\). The relations are the following:
\begin{itemize}
\item \(X^*(u, v + v') = X^*(u, v)\, X^*(u, v')\);
\item \(X^*(u, v)\, X^*(u', v')\, X^*(u, v)^{-1} = X^*(T(u, v)\, u', T(u, v)\, v')\);
\item \(X^*(u, va) = X^*(v, -ua)\) if \((u, v)\) lies in the orbit of \((e_1, e_2)\) under the action of \(\orth(M, q)\);
\item \(X^*(u, ua) = 1\).
\end{itemize}
Then the canonical morphism \(\storth(M, q) \to \storth^*(M, q), x_{ij}(a) \mapsto X^*(e_i, e_{-j} a), x_j(m) \mapsto X^*(e_{-j}, -m)\) is an isomorphism, the preimage of \(X^*(u, v)\) is the element \(X(u, v)\) from theorem \ref{esd-existence}. Also we may use \(\eorth(M, q)\) instead of \(\orth(M, q)\) in the statement.
\end{theorem}
\begin{proof}
Let \(F \colon \storth(M, q) \to \storth^*(M, q)\) be the homomorphism from the statement. By theorems \ref{esd-identities} and \ref{esd-existence}, there is a homomorphism \(G \colon \storth^*(M, q) \to \storth(M, q), X^*(u, v) \mapsto X(u, v)\). Clearly, \(G \circ F = \id\).

The group \(\orth(M, q)\) (or \(\eorth(M, q)\)) acts on \(\storth^*(M, q)\) by \(\up g{X^*(u, v)} = X^*(gu, gv)\). Clearly, \(\storth^*(M, q)\) is generated by the image of \(F\) and its conjugates under this action (here we need that \(u\) lies in the orbit of \(e_1\)). Since \(\storth(M, q)\) is perfect, it follows that \(\storth^*(M, q)\) is perfect. The canonical homomorphism \(\storth^*(M, q) \to \orth(M, q), X^*(u, v) \mapsto T(u, v)\) has central kernel by the second relation, hence the kernel of \(G\) is also central.

Now \(G \colon \storth^*(M, q) \to \storth(M, q)\) is a split perfect central extension. It is well-known that such a homomorphism is necessarily an isomorphism, its splitting \(F\) is the inverse.
\end{proof}

\bibliographystyle{plain}  
\bibliography{references}

\end{document}